\documentclass[12pt,twoside]{amsart} 
\usepackage{amssymb, amsmath, amsthm, mathrsfs} 
\usepackage{url} 
\usepackage[dvipdfmx]{graphicx}

\def\house#1{{%
    \setbox0=\hbox{$#1$}
    \vrule height \dimexpr\ht0+1.4pt width .4pt depth \dp0\relax
    \vrule height \dimexpr\ht0+1.4pt width \dimexpr\wd0+2pt depth \dimexpr-\ht0-1pt\relax
    \llap{$#1$\kern1pt}
    \vrule height \dimexpr\ht0+1.4pt width .4pt depth \dp0\relax
}}

\title[Northcott numbers for the weighted Weil heights]
{Northcott numbers for the weighted Weil heights}
\author[M. Okazaki - K. Sano]
{Masao Okazaki, Kaoru Sano}
\keywords{weighted Weil height, Northcott property, Bogomolov property, Northcott number}
\subjclass[2020]{Primary 11G50} 

\address[M. Okazaki]{Graduate School of Mathematics, Kyushu University, Motooka 744, Nishi-ku, Fukuoka 819-0395, Japan.}
\email{m-okazaki@math.kyushu-u.ac.jp}

\address[K. Sano]{Faculty of Science and Engineering, Doshisha University, Kyoto, 610-0394, Japan.}
\email{kaosano@mail.doshisha.ac.jp}

\newenvironment{parts}[0]{%
  \begin{list}{}%
    {\setlength{\itemindent}{0pt}
     \setlength{\labelwidth}{1.5\parindent}
     \setlength{\labelsep}{.5\parindent}
     \setlength{\leftmargin}{2\parindent}
     \setlength{\itemsep}{0pt}
     }%
   }%
  {\end{list}}
\newcommand{\Part}[1]{\item[\upshape#1]}

\newtheorem{thm}{Theorem}[section]
\newtheorem{lem}[thm]{Lemma}
\newtheorem{cor}[thm]{Corollary}
\newtheorem{prop}[thm]{Proposition}

\theoremstyle{definition}
\newtheorem{defn}[thm]{Definition}
\newtheorem{ques}[thm]{Question}

\newtheorem{rem}[thm]{Remark}
\newtheorem*{ack}{Acknowledgments}

\newcommand{\Q}{\mathbb{Q}} 
\newcommand{\R}{\mathbb{R}} 
\newcommand{\Z}{\mathbb{Z}} 
\newcommand{\QB}{\overline{\mathbb{Q}}} 
\newcommand{\Qtr}{\mathbb{Q}^{\textrm{tr}}} 
\newcommand{\h}{\mathfrak{h}} 
\newcommand{\e}{\varepsilon} 
\newcommand{\g}{\gamma}

\newcommand{\ie}{\textit{i.e.,~}} 
\newcommand{\eg}{\textit{e.g.,~}} 
\newcommand{\f}{\frac} 

\newcommand{\N}{\mathbb{N}}
\renewcommand{\d}{\delta} 
\renewcommand{\l}{\left} 
\renewcommand{\r}{\right} 
 
\renewcommand{\O}{\mathcal{O}} 
\renewcommand{\t}{\text} 

\DeclareMathOperator{\Nor}{Nor}

\begin{document} 
\begin{abstract} 
We answer the question of Vidaux and Videla about the distribution of the Northcott numbers for the Weil height. 
We solve the same problem for the weighted Weil heights. 
These heights generalize both the absolute and relative Weil height. 
Our results also refine those of Pazuki, Technau, and Widmer. 
\end{abstract} 

\maketitle

\section{Introduction}\label{Intro} 

For a subset $A\subset\QB$, a function $\h:\QB\rightarrow\R_{\geq0}$, and a real number $C>0$, we set 
\begin{align*} 
B(A,\h,C)&:=\l\{ a\in A \mid \h(a)<C \r\} \t{ and} \\
Z(A,\h)&:=\l\{ a\in A \mid \h(a)=0 \r\} 
\end{align*} 
The following definitions are important concepts when we study finiteness properties in number theory. 

\begin{defn}[\cite{BZ} and \cite{PTW2}]\label{NorthBogo} 
\ 

\begin{parts} 
\Part{(1)} 
We say that $A$ has the \textit{$\h$-Northcott property} (or $\h$-(N) for short) if the set $B(A,\h,C)$ is finite for all $C>0$. 

\Part{(2)} 
We say that $A$ has the \textit{$\h$-Bogomolov property} (or $\h$-(B) for short) if the set $B(A,\h,C)\setminus Z(A,\h)$ is finite for some $C>0$. 
\end{parts} 
\end{defn} 

Definition \ref{NorthBogo} lets us set 
\[
\Nor_\h(A):=\inf\l\{ C>0 \mid \# B(A,\h,C)=\infty \r\}. 
\] 
The non-negative number $\Nor_\h(A)$ is called the \textit{Northcott number of $A$ with respect to $\h$}, introduced in \cite{PTW2} and \cite{VV}. 
By definition, a subset $A$ has $\h$-(N) (resp. $\h$-(B)) if and only if $\Nor_\h(A)=\infty$ (resp. $\Nor_\h(A\setminus Z(A,\h))>0$). 
Here we regard $\inf\emptyset$ as $\infty$. 
We note that $\h$-(B) immediately follows from $\h$-(N). 

Let $h:\QB\rightarrow\R_{\geq0}$ be the absolute logarithmic Weil height. 
In \cite{VV}, Vidaux and Videla proposed the following question. 

\begin{ques}[{\cite[Question 6]{VV}}]\label{Taisetsunamondai} 
Which real numbers can be realized as $\Nor_h(L)$ for some field $L\subset\QB$? 
\end{ques} 

The above question was first dealt with in \cite[Theorem 3]{PTW2}, which reveals that for any given $c>0$ there exists a field $L\subset\QB$ satisfying that $c/2\leq \Nor_h(L)\leq c$. 
This result lets us expect that any positive real number can be realized as $\Nor_h(L)$ for some field $L\subset\QB$. 
Hence our main result is the following answer to Question \ref{Taisetsunamondai}. 

\begin{thm}\label{Answer} 
For any given $c>0$, we can construct a field $L$ satisfying that ${\rm Nor}_h(L)=c$. 
\end{thm}

We prove a more general result for the weighted Weil heights, which generalize the absolute and the relative Weil height. 
For each $\g\in\R$ and $a\in\QB$, we set 
\[
h_\g(a):=\deg(a)^\g h(a), 
\]
where $\deg(a):=[\Q(a):\Q]$. 
The function $h_\g$ is called the {\it $\g$-weighted Weil height}, introduced in \cite{PTW2}. 
Here we should note that $h_\g$-(B) was studied implicitly in \cite{AM} before \cite{PTW2}. 
We first remark that $h_0$ (resp. $h_1$) is the absolute (resp. relative) Weil height. 
We also note that the equality $Z(\QB, h_\g)=\mu_{\QB}\cup\{0\}$ holds for all $\g\in\R$, where $\mu_A$ is the set of roots of unity in a subset $A\subset\QB$ (see, \eg \cite[Theorem 1.5.9]{BG}). 
We denote $h_\g$-(N) (resp. $h_\g$-(B)) by $\g$-(N) (resp. $\g$-(B)) for short. 
We write $\Nor_{h_\g}(\cdot)$ as $\Nor_\g(\cdot)$ for simplicity. 
The property $0$-(N) (resp. $0$-(B)) is usually called the \textit{Northcott property} (resp. \textit{Bogomolov property}). 
There are several examples of infinite extensions of $\Q$ that have $0$-(N) or $0$-(B) 
(see, \eg \cite{AD}, \cite{BZ}, \cite{Hab}, \cite{Schi}, or \cite{Wi}). 
On the other hand, the Lehmer conjecture asserts that $\QB$ has $1$-(B) and $\Nor_1(\QB\setminus\mu_{\QB})\geq \log(s)$, where $s=1.176\ldots$ is the smallest known Salem number (see, \eg \cite[p.100]{Si07}). 
In \cite[Theorem 4]{PTW2}, Pazuki, Technau, and Widmer gave fields that have $1$-(B) but not $0$-(B). 
More precisely, for each $\g\leq 1$ and $\epsilon>0$, they constructed a field $L\subset\QB$ that has $\g$-{\rm (N)} but not $(\g-\epsilon)$-{\rm (B)}. 
As we will see in Lemma \ref{intervalization}, the sets 
\begin{align*} 
I_N(A)&:=\l\{ \g\in\R \mid A \t{ has } \g\t{-(N)} \r\} \t{ and} \\
I_B(A)&:=\l\{ \g\in\R \mid A \t{ has } \g\t{-(B)} \r\} 
\end{align*} 
are intervals of the form $(\g,\infty)$ or $[\g,\infty)$ for some $\g\in\R\cup\{\pm\infty\}$. 
Thus the field $L$ in \cite[Theorem 4]{PTW2} satisfies that $(\g-\epsilon,\infty)\supset I_B(L) \supset I_N(L)\supset [\g,\infty)$. 
However, the values of $\inf I_B(L)$ and $\inf I_N(L)$ were not mentioned in \cite[Theorem 4]{PTW2}: 

\begin{figure}[h]
\center{\includegraphics[width=11.4345cm, height=1.7787cm]{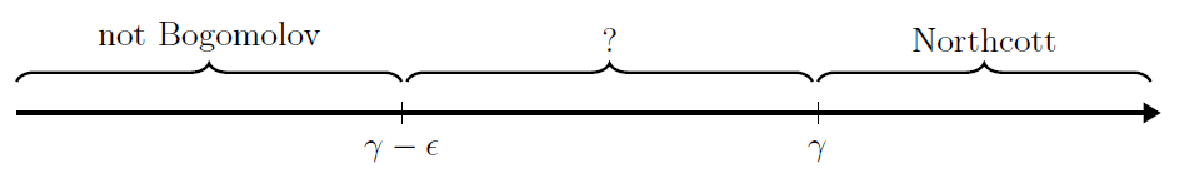}}
\end{figure} 

As we will see in Corollary \ref{inf} and Remark \ref{QB}, we remark that 
\[ 
\inf I_N(A\setminus\mu_A)=\inf I_B(A)\in[-\infty,1]
\] 
hold for all subsets $A\subset\QB$. 
Hence our generalization of Theorem \ref{Answer} is the following. 

\begin{thm}\label{main} 
Let $\g<1$ and $c>0$ be real numbers. 
We can construct a field $L\subset\QB$ satisfying the following condition {\rm (1)}, {\rm (2)}, or {\rm (3)}, respectively. 
\begin{parts} 
\Part{(1)} 
$I_N(L)=I_B(L)=[\g,\infty)$. 
\Part{(2)} 
$I_N(L)=(\g,\infty)\subsetneq[\g,\infty)=I_B(L)$ with ${\rm Nor}_\g(L)=c$. 
\Part{(3)} 
$I_N(L)=I_B(L)=(\g,\infty)$. 
\end{parts} 
\end{thm} 

\begin{rem}\label{Lehmer} 
We note that Theorem \ref{main} covers neither the case $\g=1$ nor $\g=-\infty$. 
We will discuss fields $L$ satisfying each of $I_N(L)=I_B(L)=[1,\infty)$, $I_N(L)=(1,\infty)\subsetneq[1,\infty)=I_B(L)$, and $I_N(L)=I_B(L)=\R$ in Section \ref{Supp}. 
\end{rem} 

\begin{rem} 
There are some works on calculating Northcott numbers with respect to the house. 
The notion of Northcott numbers originally stems from the Julia-Robinson numbers, which are Northcott numbers of sets of totally positive algebraic integers with respect to the house (see, \eg \cite{CVV}, \cite{GR}, \cite{Rob}, \cite{VV2}, or \cite{VV}). 
Another remarkable result is \cite[Theorem 1]{PTW2}. 
Pazuki, Technau, and Widmer constructed for any given $c\geq1$ a field whose ring of integers has Northcott number (with respect to the house) equal to $c$.
The image under the house of their rings of integers also satisfies additional topological properties motivated by work on the undecidability of rings (see, \eg \cite{Rob} and \cite{VV2}). 
\end{rem} 

The other result concerns the following theorem. 
\begin{thm}[\eg \t{\cite[Theorem 2.1]{DZ}}]\label{North} 
Let $L\subset\QB$ be a field with $0${\rm -(N)}. 
Then for any $d>0$, the set 
\[
L^{(d)}:=\l\{ a\in\QB \mid [L(a):L]\leq d \r\}
\] 
also has $0${\rm -(N)}. 
\end{thm} 

The well-known fact that any number field has $0$-(N) is an immediate consequence of Theorem \ref{North}. 
We show that Theorem \ref{North} fails for $h_\g$ whenever $0<\g\leq 1$. 

\begin{prop}\label{Qtr} 
Let $\Q^{\rm tr}$ be the field of all totally real numbers. 
Then for all $0<\g\leq1$ the field $\Q^{\rm tr}$ has $\g$-{\rm (N)} but the field $\Q^{\rm tr}(\sqrt{-1})$ does not have $\g$-{\rm (N)}. 
More precisely, for all $0<\g\leq1$ the set $\Q^{\rm tr}(\sqrt{-1})\setminus\mu_{\Q^{\rm tr}(\sqrt{-1})}$ does not have $\g$-{\rm (N)}. 
\end{prop} 

\subsection*{Outline of this paper} 
In Section \ref{Pre}, we study properties of $I_N(A)$ and $I_B(A)$ that we use later. 
As a consequence of Proposition \ref{weak}, we prove Proposition \ref{Qtr} at the end of Section \ref{Pre}. 
In Section \ref{Kakai}, we give some lemmata needed to calculate Northcott numbers. 
In Section \ref{proof}, we give a more explicit form of Theorem \ref{main} and prove it. 
We also compare our fields with those constructed in \cite{PTW2} and then we present further problems. 
In Section \ref{Supp}, we give some supplemental remarks on the cases $\g=1$ and $\g=-\infty$ of the conditions in Theorem \ref{main}.

\section{Intervals associated with $\g$-(N) and $\g$-(B)}\label{Pre} 

Let $A\subset\QB$. 
In this section, we give some basic properties of $I_N(A)$ and $I_B(A)$ defined in Section \ref{Intro}. 

\begin{lem}\label{intervalization} 
Let $\g\in\R$. 
If $A\subset\QB$ has $\gamma$-{\rm (N)} {\rm (}resp. $\gamma$-{\rm (B)}{\rm )}, the set $A$ also has $\delta$-{\rm (N)} {\rm (}resp. $\delta$-{\rm (B)}{\rm )} for all $\delta>\gamma$. 
\end{lem} 
\begin{proof} 
Let $\d>\g$ and $A\subset\QB$. 
Since the inequality $h_\g(a)\leq h_\d(a)$ holds for all $a\in\QB$, the property $\d$-(N) of $A$ immediately follows from $\g$-(N) of $A$. 
On the other hand, a set $A$ having $\g$-(B) means that there exists a constant $D>0$ such that $h_\g(a)\geq D$ for all $a\in A\setminus (\mu_A\cup\{0\})$. 
Thus $\d$-(B) of $A$ also follows from $\g$-(B) of $A$ since we know that $h_\d(a)\geq h_\g(a)\geq D$ for all $a\in A\setminus (\mu_A\cup\{0\})$. 
\end{proof} 

\begin{prop}\label{interval} 
For each subset $A\subset\QB$, the sets $I_N(A)$ and $I_B(A)$ are intervals of forms $(\g,\infty)$ or $[\g,\infty)$ for some $\g\in\R\cup\{\pm\infty\}$. 
\end{prop} 

\begin{proof} 
This is an immediate consequence of Lemma \ref{intervalization}. 
\end{proof} 

By the definitions of $\g$-(N) and $\g$-(B), we know that $I_N(A)\subset I_B(A)$, \ie $\inf I_N(A)\geq\inf I_B(A)$. 
We actually have $\inf I_N(A\setminus\mu_A)=\inf I_B(A)$. 
This is an immediate consequence of the following proposition. 

\begin{prop}\label{weak} 
Let $\g\in\R$. 
If $A\subset\QB$ has $\g$-{\rm (B)}, the set $A\setminus\mu_A$ has $\d$-{\rm (N)} for all $\d>\g$. 
\end{prop} 

\begin{proof} 
Take any $C>0$ and $a\in A\setminus(\mu_A\cup\{0\})$ with $h_\d(a)<C$. 
Since $A$ has $\g$-(B), there exists a constant $D>0$ independent of $a$ such that $h_\g(a)\geq D$. 
Thus we have 
\[
\deg(a)
=\l(\f{h_\d(a)}{h_\g(a)}\r)^{\f{1}{\d-\g}}
<\l(\f{C}{D}\r)^{\f{1}{\d-\g}}. 
\tag{2.1}\label{2.1} 
\] 

If $\d\geq 0$, then we have 
\[
h(a)\leq \deg(a)^\d h(a)=h_\d(a)<C. 
\tag{2.2}\label{2.2} 
\] 
By Theorem \ref{North}, the conditions (\ref{2.1}) and (\ref{2.2}) imply that the set $B(A\setminus\mu_A, h_\d,C)$ is finite for the case $\d\geq0$. 

On the other hand, if $\d<0$, then (\ref{2.1}) implies that $\deg(a)^{-\d}<(C/D)^{\f{-\d}{\d-\g}}$. 
Therefore we have 
\[
h(a)
=\deg(a)^{-\d}h_\d(a) 
<\l(\f{C}{D}\r)^{\f{-\d}{\d-\g}}C. 
\tag{2.3}\label{2.3} 
\] 
Again by Theorem \ref{North}, the conditions (\ref{2.1}) and (\ref{2.3}) also imply that the set $B(A\setminus\mu_A, h_\d,C)$ is finite for the case $\d<0$. 
\end{proof} 

\begin{cor}\label{inf} 
For any subset $A\subset\QB$, the equality $\inf I_N(A\setminus\mu_A)=\inf I_B(A)$ holds. 
In addition, if $A$ contains only finitely many roots of unity, we have $\inf I_N(A)=\inf I_B(A)$. 
\end{cor} 

\begin{rem}\label{Abel} 
We can not remove the assumption ``excluding $\mu_A$'' in Proposition \ref{weak}. 
This is because the set $\mu_{\QB}$ clearly has $\g$-(B) but not $\g$-(N) for all $\g\in\R$. 
\end{rem} 

\begin{rem}\label{QB} 
We set 
\[
h_{\rm Dob}(a):=\l(\f{\log'(\deg(a))}{\log'\log(\deg(a))}\r)^3h_1(a), 
\] 
where $\log'(\cdot):=\max\{1,\log(\cdot)\}$. 
In \cite{Do}, Dobrowolski proved that $\QB$ has $h_{\rm Dob}$-(B). 
Hence there exists $D>0$ such that $f(a)\geq D$ for all $a\in\QB\setminus(\mu_{\QB}\cup\{0\})$. 
It is clear that for any $\e>0$, if $\deg(a)$ is sufficiently large, the inequality 
\[
\deg(a)^\e>\l(\f{\log'(\deg(a))}{\log'\log(\deg(a))}\r)^3 
\] 
holds. 
Thus $\QB$ has $(1+\e)$-(B) for any $\e>0$ by Theorem \ref{North}. 
Therefore, by Corollary \ref{inf}, we have 
\[ 
I_N(\QB\setminus\mu_{\QB})=(1,\infty)
\ \t{ and } \ 
\inf I_N(A\setminus\mu_A)=\inf I_B(A)\in[-\infty,1] 
\] 
for all subsets $A\subset\QB$. 
Note that these immediately follow if we assume the Lehmer conjecture. 
\end{rem} 

As a consequence of Proposition \ref{weak}, we prove Proposition \ref{Qtr}. 

\noindent 
\begin{proof}[Proof of Proposition \ref{Qtr}] 
It is known that $\Q^{\rm tr}$ has $0$-(B) (see \cite{Schi}). 
Thus $\Qtr$ has $\g$-(N) for each $\g>0$ by the equality $\mu_{\Qtr}=\{\pm1\}$ and Proposition \ref{weak}. 
On the other hand, the field $\Qtr(\sqrt{-1})$ does not have $\g$-(N) for all $0<\g\leq 1$. 
Indeed, the algebraic number $a_k:=((2-\sqrt{-1})/(2+\sqrt{-1}))^{1/k}$ is an element of $\Qtr(\sqrt{-1})$ for each $k\in\Z_{>0}$ (see, \eg \cite[Section 5]{ASZ}). 
Since we have the inequalities 
\[
h_\g(a_k)=\deg(a_k)^\g h(a_1^{1/k})\leq \f{2^\g}{k^{1-\g}}h(a_1) \leq 2h(a_1), 
\] 
the field $\Qtr(\sqrt{-1})$ has infinitely many elements with bounded value of $h_\g$. 
\end{proof}

\section{Some remarks on lower bounds for heights}\label{Kakai} 

Throughout the rest of the paper, we denote the set of positive integers by $\N$. 
This section is devoted to giving some technical lemmata to calculate the $\g$-Northcott numbers. 
Theorem \ref{Siiq} and Lemma \ref{NNumber} allow us to get a lower bound for the Northcott number of our fields. 

\begin{thm}\label{Siiq} 
Let $K$ be a number field. 
Assume that $a\in \QB$ satisfies that $[K(a):K]>1$. 
We set $M:=K(a)$ and $m:=[M:K]$. 
Then we have the inequality 
\[
h(a)\geq \f{1}{2(m-1)}\l(\f{\log(N_{K/\Q}(D_{M/K}))}{m[K:\Q]}-\log(m)\r), 
\] 
where $N_{K/\Q}$ is the usual norm and $D_{M/K}$ is the relative discriminant ideal of the extension $M/K$. 
\end{thm} 
\begin{proof} 
See \cite[Theorem 2]{Si84}. 
\end{proof} 

\begin{lem}\label{NNumber} 
Let $\g\in\R$ and $A\subset\QB$. 
We set 
\[
\delta_\g(B):=\inf h_\g(B)
\]
for each non-empty subset $B\subset A$. 
Let  $A_0\subsetneq A_1\subsetneq A_2\subsetneq\cdots$ be an ascending chain of non-empty subsets of $A$ satisfying that 
\begin{parts} 
\Part{(1)} 
$A_i$ has $\g$-{\rm(N)} for all $i\in\Z_{\geq 0}$ and 

\Part{(2)} 
$A=\bigcup_{i\in\Z_{\geq0}}A_i$. 

\end{parts} 
Then we have 
\[
{\rm Nor}_\g(A)=\liminf_{i\rightarrow\infty}\delta_\g(A_i\setminus A_{i-1}). 
\]
\end{lem} 

\begin{proof} 
See \cite[Lemma 6]{PTW2}. 
\end{proof} 

\begin{prop}\label{Criteria} 
Let $\g\leq1$ be a real number. 
We also let $(p_i)_{i\in\N}$, $(q_i)_{i\in\N}$, and $(d_i)_{i\in\N}$ be strictly increasing sequences of prime numbers. 
Assume that there exists $i_0\in\N$ such that for all $i>i_0$, the inequality $p_i<q_i$ and 
\[ 
p_i,q_i\notin\{d_1,p_1,q_1,\ldots,d_{i-1},p_{i-1},q_{i-1}\}
\] 
hold. 
We set 
\[
V(i,\g):=
\begin{cases} 
\displaystyle \log(p_i)-\f{\log(d_i)}{2} & (\g=1), \\ 
\displaystyle \f{\log(p_i)}{d_i^{1-\g}} & (0\leq\g<1), \\ 
\displaystyle \f{\log(p_i)}{(d_1\cdots d_{i-1})^{-\g}d_i^{1-\g}} & (\g<0). 
\end{cases} 
\] 
Then the field $L:=\Q((p_i/q_i)^{1/d_i} \mid i\in\N)$ satisfies that 
\[
{\rm Nor}_\g(L)\geq\liminf_{i\rightarrow\infty}V(i,\g). 
\] 
Especially, we have the following. 
\begin{parts} 
\Part{(1)} 
$L$ has $\g$-{\rm (N)} if $\liminf_{i\rightarrow\infty}V(i,\g)=\infty$. 
\Part{(2)} 
$L$ has $\g$-{\rm (B)} if $\liminf_{i\rightarrow\infty}V(i,\g)>0$. 
\end{parts} 
\end{prop} 

\begin{proof} 
We set 
\[
K_0:=\Q, \ 
K_i:=K_{i-1}((p_i/q_i)^{1/d_i}), \ \t{and} \ 
F_i:=\Q((p_i/q_i)^{1/d_i})
\] 
for each $i\in\N$. 
By Lemma \ref{NNumber}, it is enough to show that the inequality 
\[
\liminf_{i\rightarrow\infty}\delta_\g(K_i\setminus K_{i-1})\geq \liminf_{i\rightarrow\infty}V(i,\g) 
\label{3.1}\tag{3.1} 
\] 
holds. 
Take any $i>i_0$ and $a\in K_i\setminus K_{i-1}$. 
Note that the equality $K_{i-1}(a)=K_i$ holds since $[K_i:K_{i-1}]=d_i$ is a prime number and $a\notin K_{i-1}$. 
Thus we have the inequalities 
\[
d_i\leq \deg(a) \leq d_1\cdots d_i
\label{3.2}\tag{3.2} 
\] 
and 
\[ 
h_\g(a) 
\geq \f{\deg(a)^\g}{2(d_i-1)}\l(\f{\log(N_{K_{i-1}/\Q}(D_{K_i/K_{i-1}}))}{d_i[K_{i-1}:\Q]}-\log (d_i)\r) 
\label{3.3}\tag{3.3} 
\] 
by Theorem \ref{Siiq}. 
To estimate $h_\g(a)$ from below, we give a lower bound for $N_{K_{i-1}/\Q}(D_{K_i/K_{i-1}})$. 
Since $F_i=\Q((p_iq_i^{d_i-1})^{1/d_i})$ and $(p_iq_i^{d_i-1})^{1/d_i}$ is a root of the $p_i$-Eisenstein polynomial $X^{d_i}-p_id_i^{d_i-1}\in\Q[X]$, we know that $p_i$ ramifies totally in $F_i$ (see, \eg \cite[Theorem 24 (a)]{FT}). 
Thus it holds that $p_i^{d_i-1} \mid D_{F_i/\Q}$ (see, \eg \cite[p.199, (2.6)]{Ne} and \cite[p.201, (2.9)]{Ne}). 
Now we note that the equalities 
\[ 
N_{K_{i-1}/\Q}(D_{K_i/K_{i-1}})D_{K_{i-1}/\Q}^{[K_i:K_{i-1}]}
=D_{K_i/\Q}
=N_{F_i/\Q}(D_{K_i/F_i})D_{F_i/\Q}^{[K_i:F_i]} 
\label{3.4}\tag{3.4}
\] 
hold (see, \eg \cite[p.202, (2.10)]{Ne}). 
Since $K_{i-1}$ is the compositum of $F_1$,\,$\ldots$\,, $F_{i-1}$ and $p_i\notin\{d_1,p_1,q_1,\ldots,d_{i-1},p_{i-1},q_{i-1}\}$, we know that $p_i$ does not ramify in $K_{i-1}$ (see, \eg \cite[Theorem 85]{Hi} and \cite[Lemma 4.1]{Fi}). 
Hence $p_i\nmid D_{K_{i-1}/\Q}$ holds. 
Therefore (\ref{3.4}) yields that 
\[ 
p_i^{[K_i:F_i](d_i-1)} \mid N_{K_{i-1}/\Q}(D_{K_i/K_{i-1}}). 
\label{3.5}\tag{3.5}
\] 
Replacing $p_i$ with $q_i$ in the above discussion, we also get 
\[ 
q_i^{[K_i:F_i](d_i-1)} \mid N_{K_{i-1}/\Q}(D_{K_i/K_{i-1}}). 
\label{3.6}\tag{3.6}
\] 
Since the equality $[K_i:F_i]=d_1\cdots d_{i-1}=[K_{i-1}:\Q]$ holds, the conditions (\ref{3.5}) and (\ref{3.6}) yield that 
\[
2\log(p_i) \leq \log(p_iq_i) \leq \f{\log(N_{K_{i-1}/\Q}(D_{K_i/K_{i-1}}))}{[K_{i-1}:\Q](d_i-1)}. 
\label{3.7}\tag{3.7}
\] 
Thus we conclude that 
\begin{align*} 
&~h_\g(a) \\
\geq&~ \deg(a)^\g\l(\f{\log(p_i)}{d_i}-\f{\log(d_i)}{2(d_i-1)}\r) && \text{by }(\ref{3.3}), (\ref{3.7}) \\
=&~\l(\f{\deg(a)}{d_i}\r)^\g\l(\f{\log(p_i)}{d_i^{1-\g}}-\f{d_i^\g\log(d_i)}{2(d_i-1)}\r) \\
\geq&~ 
\begin{cases}  
\displaystyle \f{\log(p_i)}{d_i^{1-\g}}-\f{d_i^\g\log(d_i)}{2(d_i-1)} & (0\leq\g\leq1) \\ 
\displaystyle (d_1\cdots d_{i-1})^\g\l(\f{\log(p_i)}{d_i^{1-\g}}-\f{d_i^\g\log(d_i)}{2(d_i-1)}\r) & (\g<0)
\end{cases}
&& \text{by }(\ref{3.2}) \\
=&~ 
\begin{cases} 
\displaystyle V(i,1)-\f{\log(d_i)}{2(d_i-1)} & (\g=1), \\
\displaystyle V(i,\g)-\f{d_i^\g\log(d_i)}{2(d_i-1)} & (0\leq\g<1), \\ 
\displaystyle V(i,\g)-\f{\log(d_i)}{2(d_1\cdots d_i)^{-\g}(d_i-1)} & (\g<0). 
\end{cases} 
\end{align*} 
Since we have the equalities 
\begin{align*} 
&\lim_{i\rightarrow\infty}\f{\log(d_i)}{2(d_i-1)}=0, \\
&\lim_{i\rightarrow\infty}\f{d_i^\g\log(d_i)}{2(d_i-1)}=0 && \hspace{-1.5cm} (0\leq\g<1), \t{ and} \\
&\lim_{i\rightarrow\infty}\f{\log(d_i)}{2(d_1\cdots d_i)^{-\g}(d_i-1)}=0 && \hspace{-1.5cm} (\g<0), 
\end{align*} 
we get the inequality (\ref{3.1}) by letting $i\rightarrow\infty$. 
\end{proof}

\section{Controlling Northcott numbers}\label{proof} 

\subsection{Proof of Theorem \ref{main}} 

In this section, we prove Theorem \ref{main}. 
More precisely, we prove the following theorem. 

\begin{thm}\label{main'} 
Let $\g<1$ and $c>0$ be real numbers. 
We also let $(p_i)_{i\in\N}$, $(q_i)_{i\in\N}$, and $(d_i)_{i\in\N}$ be strictly increasing sequences of prime numbers with $q_i<p_{i+1}$ and $p_i< q_i <2p_i$ for all $i\in\N$. 
Furthermore, we let $f(x)$ be $\log(x)$, $c$, or $1/\log(x)$. 
If $\g\geq0$, then we assume that the inequalities 
\[
\exp(f(d_i)d_i^{1-\g}) \leq p_i \leq 2\exp(f(d_i)d_i^{1-\g}) 
\] 
hold for all $i\in\N$. 
On the other hand, if $\g<0$, then we assume that the inequalities 
\[
\exp(f(d_i)(d_1\cdots d_{i-1})^{-\g}d_i^{1-\g}) \leq p_i \leq 2\exp(f(d_i)(d_1\cdots d_{i-1})^{-\g}d_i^{1-\g}) 
\] 
hold for all $i\in\N$ and the equality 
\[
\lim_{i\rightarrow\infty}\f{i\log(d_i)}{d_i^{-\g}}=0 
\] 
holds. 
We set $L:=\Q((p_i/q_i)^{1/d_i} \mid i\in\N)$. 
\begin{parts} 
\Part{(1)} 
If $f(x)=\log(x)$, the field $L$ satisfies that $I_N(L)=I_B(L)=[\g,\infty)$. 
\Part{(2)} 
If $f(x)=c$, the field $L$ satisfies that $I_N(L)=(\g,\infty)\subsetneq[\g,\infty)=I_B(L)$ with ${\rm Nor}_\g(L)=c$. 
\Part{(3)} 
If $f(x)=1/\log(x)$, the field $L$ satisfies that $I_N(L)=I_B(L)=(\g,\infty)$. 
\end{parts} 
\end{thm} 

\begin{rem} 
By the theorem of Bertrand-Chebyshev, we can take $(p_i)_{i\in\N}$ and $(q_i)_{i\in\N}$ in Theorem \ref{main'}. 
Furthermore, for each $\g<0$, the sequence of prime numbers $(d_i)_{i\in\N}$ satisfies the equality 
\[
\lim_{i\rightarrow\infty}\f{i\log(d_i)}{d_i^{-\g}}=0 
\] 
if, for example, the inequality $d_i^{-\g}\geq i^2$ holds for each $i\in\N$. 
\end{rem} 

\begin{proof}[Proof of Theorem \ref{main'}]
Let $V(i,\g)$ be the quantity defined in Proposition \ref{Criteria}. 
First, we prove the assertion in the case $\g\geq 0$. 
We prepare some notation. 
By the assumptions $\exp(f(d_i)d_i^{1-\gamma})\leq p_i$ and $\g<1$, there exists $i_0\in\N$ such that $d_i<p_i$ for all $i>i_0$. 
Combining the assumption $q_i<p_{i+1}$, we know that 
\[ 
p_i, q_i\notin\{d_1,p_1,q_1,\ldots,d_{i-1},p_{i-1},q_{i-1}\} 
\] 
for all $i>i_0$. 
Thus we can apply Proposition \ref{Criteria} to our setting. 
For each $\d\geq\g$ with $\d<1$ and $\e\leq\g$, we have the inequalities 
\[ 
V(i,\d)
\geq \f{f(d_i)d_i^{1-\g}}{d_i^{1-\d}} 
= d_i^{\d-\g}f(d_i) 
=: L_1(i,\d) 
\]
and 
\[
0<h_\e((p_i/q_i)^{1/d_i}) 
= d_i^{\e-1}\log(q_i) 
< \f{\log(4)}{d_i^{1-\e}}+\f{f(d_i)}{d_i^{\g-\e}} 
=: U_1(i,\e) 
\] 
by the assumption $q_i<2p_i\leq 4\exp(f(d_i)d_i^{1-\g})$. 

\begin{parts}
\Part{(1)} 
For $\d\geq\g$, since we have 
\[
V(i,\d)
\geq L_1(i,\d) 
= d_i^{\d-\g}\log(d_i) 
\rightarrow\infty 
\] 
as $i\rightarrow\infty$, the field $L$ has $\d$-(N) by Proposition \ref{Criteria}. 

For $\e<\g$, because we observe that 
\[
0
< h_\e((p_i/q_i)^{1/d_i})
< U_1(i,\e) 
= \f{\log(4)}{d_i^{1-\e}}+\f{\log(d_i)}{d_i^{\g-\e}} 
\rightarrow 0
\]
as $i\rightarrow\infty$, the field $L$ does not have $\e$-(B).

\Part{(2)} 
Proposition \ref{Criteria} and the inequality 
\[
V(i,\g)
\geq L_1(i,\g) 
=c
\] 
imply that $\Nor_\g(L)\geq c$ holds. 
On the other hand, since 
\[
h_\g((p_i/q_i)^{1/d_i})
< U_1(i,\g) 
= \f{\log(4)}{d_i^{1-\g}}+c 
\rightarrow c
\]
as $i\rightarrow\infty$, the inequality $\Nor_\g(L)\leq c$ also holds. 
Hence we get the equality $\Nor_\g(L)=c$. 
It also follows that $I_N(L)=(\g,\infty)\subsetneq[\g,\infty)=I_B(L)$ by Proposition \ref{weak}.

\Part{(3)} 
For $\d>\g$, since we have 
\[
V(i,\d)
\geq L_1(i,\d) 
= \f{d_i^{\d-\g}}{\log(d_i)} 
\rightarrow\infty 
\] 
as $i\rightarrow\infty$, the field $L$ has $\d$-(N) by Proposition \ref{Criteria}. 

For $\e\leq\g$, because we observe that 
\[
0
< h_\e((p_i/q_i)^{1/d_i})
< U_1(i,\e) 
= \f{\log(4)}{d_i^{1-\e}}+\f{1}{d_i^{\g-\e}\log(d_i)} 
\rightarrow 0
\]
as $i\rightarrow\infty$, the field $L$ does not have $\e$-(B). 
\end{parts}

Next, we prove the assertion in the case $\g<0$. 
As in the case $\g\geq0$, we can apply Proposition \ref{Criteria} to our setting. 
For each negative real numbers $\d\geq\g$ and $\e\leq\g$, we have the inequalities 
\[ 
V(i,\d)
\geq \f{f(d_i)(d_1\cdots d_{i-1})^{-\g}d_i^{1-\g}}{(d_1\cdots d_{i-1})^{-\d}d_i^{1-\d}} 
= (d_1\cdots d_i)^{\d-\g}f(d_i) 
=: L_2(i,\d) 
\]
and 
\begin{align*} 
0
&< h_\e((p_1/q_1)^{1/d_1}\cdots(p_i/q_i)^{1/d_i}) \\
&= (d_1\cdots d_i)^\e h((p_1/q_1)^{1/d_1}\cdots(p_i/q_i)^{1/d_i}) \\
&\leq (d_1\cdots d_i)^\e \l(h((p_1/q_1)^{1/d_1})+\cdots+h((p_i/q_i)^{1/d_i})\r) \\
&= (d_1\cdots d_i)^\e \l(\f{\log(q_1)}{d_1}+\cdots+\f{\log(q_i)}{d_i}\r) \\
&\leq (d_1\cdots d_i)^\e \l(\l(\f{1}{d_1}+\cdots+\f{1}{d_i}\r)\log(4) + \l(\f{f(d_1)}{d_1^\g}+\cdots+\f{f(d_i)}{(d_1\cdots d_i)^\g}\r)\r) \\ 
&< \f{i\log(4)}{d_i^{-\e}} + \f{f(d_1)+\cdots+f(d_{i-1})}{d_i^{-\e}} + \f{f(d_i)}{(d_1\cdots d_i)^{\g-\e}} \\
&=:U_2(i,\e) 
\end{align*} 
by the assumption $q_i<2p_i\leq 4\exp(f(d_i)(d_1\cdots d_{i-1})^{-\g}d_i^{1-\g})$. 

\begin{parts}
\Part{(1)} 
For $\d\geq\g$, since we have 
\[
V(i,\d)
\geq L_2(i,\d) 
=(d_1\cdots d_i)^{\d-\g}\log(d_i) 
\rightarrow\infty 
\] 
as $i\rightarrow\infty$, the field $L$ has $\d$-(N) by Proposition \ref{Criteria}. 

For $\e<\g$, because we observe that 
\begin{align*} 
0 
&< h_\e((p_1/q_1)^{1/d_1}\cdots(p_i/q_i)^{1/d_i}) \\
&< U_2(i,\e) \\
&< \f{i\log(4)}{d_i^{-\e}} + \f{i\log(d_i)}{d_i^{-\e}} + \f{\log(d_i)}{(d_1\cdots d_i)^{\g-\e}} \\
& \rightarrow 0 
\end{align*} 
as $i\rightarrow\infty$, the field $L$ does not have $\e$-(B).

\Part{(2)} 
Proposition \ref{Criteria} and the inequality 
\[
V(i,\g)
\geq L_2(i,\g) 
=c, 
\] 
imply that $\Nor_\g(L)\geq c$ holds. 
On the other hand, since 
\begin{align*} 
h_\g((p_1/q_1)^{1/d_1}\cdots(p_i/q_i)^{1/d_i}) 
<&~ U_2(i,\g) \\ 
<&~ \f{i\log(4)}{d_i^{-\g}} + \f{ic}{d_i^{-\g}} + c \\ 
\rightarrow&~ c 
\end{align*} 
as $i\rightarrow\infty$, the inequality $\Nor_\g(L)\leq c$ also holds. 
Hence we get the equality $\Nor_\g(L)=c$.

\Part{(3)} 
For $\d>\g$, since we have 
\[
V(i,\d)
\geq L_2(i,\d) 
=\f{(d_1\cdots d_i)^{\d-\g}}{\log(d_i)} 
\rightarrow\infty 
\] 
as $i\rightarrow\infty$, the field $L$ has $\d$-(N) by Proposition \ref{Criteria}. 

For $\e\leq\g$, because we observe that 
\begin{align*} 
0 
&< h_\e((p_1/q_1)^{1/d_1}\cdots(p_i/q_i)^{1/d_i}) \\
&< U_2(i,\e) \\
&< \f{i\log(4)}{d_i^{-\e}} + \f{i}{d_i^{-\e}\log(d_1)} + \f{1}{(d_1\cdots d_i)^{\g-\e}\log(d_i)} \\
& \rightarrow 0 
\end{align*} 
as $i\rightarrow\infty$, the field $L$ does not have $\e$-(B). 
\end{parts} 
\end{proof} 

In summary, the fields $L=\Q((p_i/q_i)^{1/d_i} \mid i\in\N)$ in Theorem \ref{main'} satisfy the following: 
\begin{table}[htb]
\centering
\caption{Stratification of $L$} 
\begin{tabular}{|c||c|}  \hline 
$f(x)$ & $I_N(L)\t{ and }I_B(L)$ \\ \hline\hline 
$\log(x)$ & $I_N(L)=I_B(L)=[\g,\infty)$ \\ \hline 
$c>0$ & $I_N(L)=(\g,\infty)\subsetneq[\g,\infty)=I_B(L)$ with $\Nor_\g(L)=c$ \\ \hline 
$1/\log(x)$ & $I_N(L)=I_B(L)=(\g,\infty)$ \\ \hline 
\end{tabular}
\end{table}

\subsection{Comparison with previous work and further problems}\label{Comparison} 
In this section, we compare our result with that in \cite{PTW2} and then present further problems.
The following is a variant of \cite[Theorem 3]{PTW2} and \cite[Theorem 4]{PTW2}. 

\begin{thm}\label{PTW} 
Let $\g<1$ and $c>0$ be non-negative real numbers. 
We also let $(p_i)_{i\in\N}$ and $(d_i)_{i\in\N}$ be strictly increasing sequences of prime numbers. 
Furthermore, we let $f(x)$ be $\log(x)$, $c$, or $1/\log(x)$. 
Assume that the inequalities 
\[
\exp(f(d_i)d_i^{1-\g})\leq p_i \leq 2\exp(f(d_i)d_i^{1-\g}) 
\] 
hold for all $i\in\N$. 
We set $L':=\Q(p_i^{1/d_i} \mid i\in\N)$. 
\begin{parts} 
\Part{(1)} 
If $f(x)=\log(x)$ the field $L'$ satisfies that $I_N(L')=I_B(L')=[\g,\infty)$. 
\Part{(2)} 
If $f(x)=c$ the field $L'$ satisfies that $I_N(L')=(\g,\infty)\subsetneq[\g,\infty)=I_B(L')$ with $c/2\leq {\rm Nor}_\g(L')\leq c$. 
\Part{(3)} 
If $f(x)=1/\log(x)$ the field $L'$ satisfies that $I_N(L')=I_B(L')=(\g,\infty)$. 
\end{parts} 
\end{thm} 

We can similarly prove Theorem \ref{PTW} as Theorem \ref{main'}. 
The field constructed in \cite[Theorem 4]{PTW2} is that in Theorem \ref{PTW} (2) replaced $\g$ with $\g-\epsilon/2$ for given positive real numbers $\g\leq1$ and $\epsilon\leq2\g$. 
Indeed, the field has $\g$-(N) but not $(\g-\epsilon)$-(B). 
The field $L'$ in Theorem \ref{PTW} (1), (2), or (3) also satisfies the condition about the intervals in Theorem \ref{main} (1), (2), or (3). 
However, to calculate the Northcott number, we employed one more sequence $(q_i)_{i\in\N}$. 
This idea is based on the argument in \cite[Section 1]{VW}. 
We also note that similar fields were dealt with in \cite[Proposition 1]{RT} and \cite[p.18, Example]{Rup}. 
Here we should emphasize the following advantage of fields in Theorem \ref{PTW}. 
The field $L'$ in Theorem \ref{PTW} (2) also satisfies the inequalities 
\[ 
c/2\leq \Nor_\g(\O_{L'}) \leq c, 
\] 
while it seems difficult to estimate the value of $\Nor_\g(\O_L)$ for our field $L$ in Theorem \ref{main'} (2), where $\O_F$ is the set of algebraic integers in a field $F\subset\QB$. 
In fact, combining Theorem \ref{PTW} (2) and \cite[proof of Theorem 1 (b)]{PTW2}, we can construct for any given $c>0$ a field $L'\subset\QB$ which satisfies that 
\[ 
c/2 \leq \Nor_0(L') \leq \Nor_0(\O_{L'}) \leq c
\ \t{ and } \ 
\Nor_{\log^+(\house{\cdot})}(\O_{L'})=c, 
\] 
where $\house{\cdot}:\QB\rightarrow\R_{\geq0}$ is the house and we set $\log^+(\cdot) := \max \{\log(\cdot), 0\}$.
As we did for the Weil height, for each $\g\in\R$ and $a\in\O_{\QB}$, we set 
\[ 
\log^{(\g)}(\house{a}):=\deg(a)^\g\log^+(\house{a}). 
\] 
We denote $\Nor_{\log^{(\g)}(\house{\cdot})}(\cdot)$ by $\Nor_\g^{\t{hs}}(\cdot)$. 
For each subset $A\subset\O_{\QB}$, we also set 
\begin{align*} 
	I_N^{\t{hs}}(A)&:=\l\{ \g\in\R \mid A \t{ has } \log^{(\g)}(\house{\cdot})\t{-(N)} \r\} \t{ and} \\
	I_B^{\t{hs}}(A)&:=\l\{ \g\in\R \mid A \t{ has } \log^{(\g)}(\house{\cdot})\t{-(B)} \r\}. 
\end{align*} 
Here we remark that, since the inequality $h(a)\leq\log^+(\house{a})$ holds for all $a\in\O_{\QB}$, we can similarly prove the equality $I_N^{\t{hs}}(A\setminus\mu_A)=I_B^{\t{hs}}(A)$ as Corollary \ref{inf}. 
We also note that the inequalities 
\begin{align*}
	\Nor_\g(F) &\leq \Nor_\g(\O_F) \leq \Nor_\g^{\t{hs}}(\O_F) \t{ and}\\
	\inf I_B(F) &\geq \inf I_B(\O_F) \geq \inf I_B^{\t{hs}}(\O_F)
\end{align*} 
hold for each $\g\in\R$ and each field $F\subset\QB$. 
These observations let us propose the following questions. 

\begin{ques}\label{IntegerNorth} 
Let $\g\leq1$ be a real number. 
\begin{parts} 
\Part{(1)} 
	Which real numbers can be realized as $\Nor_\g(\O_F)$ for some field $F\subset\QB$? 
\Part{(2)} 
	Can we give an example of a field $F\subset\QB$ such that
	all the values $\Nor_\g(F)$, $\Nor_\g(\O_F)$, and $\Nor_\g^{\t{hs}}(\O_F)$ are positive real numbers, and we can explicitly calculate all of them? 
\Part{(3)} For given real numbers $0\leq c_1\leq c_2 \leq c_3$,
	is there a field $F\subset\QB$ satisfying the equalities
	$\Nor_{\g}(F) = c_1$, $\Nor_{\g}(\O_F) = c_2$, and $\Nor_{\g}^{\t{hs}}(\O_F) = c_3$?
\Part{(4)} For given real numbers $1\geq \gamma_1 \geq \gamma_2 \geq \gamma_3$,
	is there a field $F\subset\QB$ satisfying the equalities $\inf I_B(F) = \gamma_1$, $\inf I_B(\O_F) =\gamma_2$, and $\inf I_B^{\t{hs}}(\O_F) = \gamma_3$?
\end{parts}
\end{ques}

\begin{rem} 
Although we furthermore dealt with the house in Question \ref{IntegerNorth}, we kept restricting our attention to the case $\g\leq1$ because of \cite[Theorem 1]{Dim}, which asserts that the inequality $\Nor_1^{\t{hs}}(\O_{\QB})\geq\log(2)/4$ holds. 
\end{rem} 

\begin{rem} 
We imposed the condition that the Northcott numbers are positive real numbers on Question \ref{IntegerNorth} (2) since we have already known the following trivial examples. 
For all $\g\in\R$, the field $\QB$ satisfies that 
$\Nor_\g(\QB)=\Nor_\g(\O_{\QB})=\Nor_\g^{\t{hs}}(\O_{\QB})=0$ and any number field $K$ does that $\Nor_\g(K)=\Nor_\g(\O_K)=\Nor_\g^{\t{hs}}(\O_K)=\infty$. 
\end{rem}

\section{Remarks on the cases $\g=1$ and $\g=-\infty$}\label{Supp} 

In this section, we partially deal with the cases $\g=1$ and $\g=-\infty$ of the conditions in Theorem \ref{main}. 
Before dealing with the case $\g=1$, we note that constructing a field $L$ satisfying the condition (2) or (3) in Theorem \ref{main} of $\g=1$ will disprove the Lehmer conjecture. 
Thus we deal with neither the case $\g=1$ of the conditions in Theorem \ref{main} (2) nor (3). 
Now we give a field $L\subset\QB$ such that $I_N(L)=I_B(L)=[1,\infty)$. 

\begin{prop}\label{g=1} 
Let $(p_i)_{i\in\N}$ and $(q_i)_{i\in\N}$ be strictly increasing sequences of prime numbers with $q_i<p_{i+1}$ and $p_i<q_i<2p_i$ for all $i\in\N$. 
Then the field $L:=\Q((p_i/q_i)^{1/p_i} \mid i\in\N)$ satisfies that $I_N(L)=[1,\infty)$. 
\end{prop} 

\begin{proof} 
Let $(d_i)_{i\in\N}:=(p_i)_{i\in\N}$ and $V(i,\g)$ be the quantity defined in Proposition \ref{Criteria}. 
Since we have 
\[ 
V(i,1) 
= \f{1}{2}\log(p_i) 
\rightarrow \infty 
\] 
as $i\rightarrow\infty$, the field $L$ has $1$-(N) by Proposition \ref{Criteria}. 

For $\e<1$, because we observe that 
\[ 
0
< h_\e((p_i/q_i)^{1/p_i}) 
= \f{\log(q_i)}{p_i^{1-\e}} 
< \f{\log(2p_i)}{p_i^{1-\e}} 
\rightarrow 0 
\] 
as $i\rightarrow\infty$, the field $L$ does not have $\e$-(B). 
\end{proof}

Next, we deal with fields $L\subset\QB$ such that $I_N(L)=(1,\infty)\subsetneq [1,\infty)=I_B(L)$. 
By using the result in \cite{Amo}, we can construct for any given $c>0$ a field $L\subset\QB$ satisfying that $c\leq \Nor_1(L)<\infty$. 
For each integer $b\geq2$ and prime number $p\geq3$, we set 
\begin{align*} 
L_{b,p}&:=\Q(\zeta_{p^i},b^{1/p^i} \mid i\in\N), \\
\langle b\rangle&:=(\textrm{the subgroup of }L_{b,p}^\times\textrm{ generated by }b), \t{ and}  \\
\sqrt{\langle b\rangle}&:=\l\{ a\in L_{b,p} \mid a^n\in\langle b\rangle \text{ for some }n\in\N \r\},  
\end{align*}
where $\zeta_m$ is a primitive $m$-th root of unity for each $m\in\N$. 
Amoroso gave the following theorem. 

\begin{thm}[{\cite[Theorem 3.3]{Amo}}]\label{Amoroso} 
If $p\nmid b$ and $p^2\nmid b^{p-1}-1$, the set $L_{b,p}\setminus\sqrt{\langle b\rangle}$ has $0$-{\rm (B)}. 
\end{thm} 

\begin{prop}\label{g=1B} 
Let $c>0$ be a real number. 
We also let a prime number $b\in\N$ satisfy that $b\geq\exp(c)$ and $b\in\l\{ 9n+2 \mid n\in\Z \r\}$.  
Then the field $L'_{b,3}:=\Q(b^{1/3^i} \mid i\in\N)$ satisfies that $I_N(L'_{b,3})=(1,\infty)\subsetneq[1,\infty)=I_B(L'_{b,3})$ and $c\leq {\rm Nor}_1(L'_{b,3})<\infty$. 
\end{prop} 

\begin{rem} 
By Dirichlet's theorem on primes in arithmetic progressions, such $b$ as in Proposition \ref{g=1B} always exists. 
\end{rem} 

\begin{proof}[Proof of Proposition \ref{g=1B}] 
By Corollary \ref{inf} and the equality $\mu_{L'_{b,3}}=\{\pm1\}$, it is sufficient only to prove the inequalities $c\leq {\rm Nor}_1(L'_{b,3})<\infty$. 
Note that $3\nmid b=9n+2$ and $3^2\nmid b^{3-1}-1=(9n+3)(9n+1)$ hold. 
Thus, by Proposition \ref{weak} and Theorem \ref{Amoroso}, the set $L'_{b,3}\setminus\sqrt{\langle b\rangle}$ has $1$-(N). 
Thus there are only finitely many $a\in L'_{b,3}\setminus\sqrt{\langle b\rangle}$ such that $h_1(a)\leq \log(b)$. 
On the other hand, we know that $h_1(a)\geq \log(b)$ for all $a\in\sqrt{\langle b\rangle}\setminus\{\pm1\}$ and that  $h_1(b^{1/3^i})=\log(b)$ for all $i\in\N$. 
Hence we have $\Nor_1(L'_{b,3})=\log(b)\geq c$. 
\end{proof}

Finally, we deal with fields $L\subset\QB$ such that $I_N(L)=I_B(L)=\R$. 
We remark that any number field satisfies the condition by Theorem \ref{North}. 
Here we give such a field of infinite extensions of $\Q$. 

\begin{prop}\label{g=-infty} 
Let $(d_i)_{i\in\N}$, $(p_i)_{i\in\N}$, and $(q_i)_{i\in\N}$ be strictly increasing sequences of prime numbers.
Assume that the inequalities 
\[ 
\exp(d_i^{1+i^2})\leq p_i<q_i<p_{i+1}
\] 
hold for all $i\in\N$. 
Then the field $L:=\Q((p_i/q_i)^{1/d_i} \mid i\in\N)$ satisfies that $I_N(L)=\R$. 
\end{prop} 

\begin{proof} 
Let $V(i,\g)$ be the quantity defined in Proposition \ref{Criteria}. 
For $\g<0$, since we have 
\[ 
V(i,\g) 
\geq \f{d_i^{1+i^2}}{(d_1\cdots d_{i-1})^{-\g}d_i^{1-\g}} 
\geq \f{d_i^{1+i^2}}{d_i^{1-i\g}} 
= d_i^{i(i+\g)}
\rightarrow \infty 
\] 
as $i\rightarrow\infty$, the field $L$ has $\g$-(N) by Proposition \ref{Criteria}. 
\end{proof}

\begin{ack} 
The authors would like to thank professor Pazuki, Technau, and Widmer for their favorable comments. 
The valuable comments of professor Technau made many sentences in the draft better. 
The suggestion of professor Widmer highly improved the proof of Proposition \ref{Criteria}. 
The authors express great gratitude to professor Toshiki Matsusaka for suggesting their joint work. 
Without his suggestion, this work would not exist. 
The authors are grateful to professor Masanobu Kaneko for reading the draft carefully and pointing out some errata. 
The first author thanks his doctoral advisor Yuichiro Takeda for introducing \cite{PTW2} to him. 
The authors greatly appreciate the anonymous referee for many valuable comments. 
These greatly improved the whole draft. 
The first author was supported until March 2022 by JST SPRING, Grant Number JPMJSP2136. 
The second author is supported by JSPS KAKENHI Grant Number JP20K14300. 
\end{ack}

\end{document}